\documentclass{article}
\usepackage{dsfont}
\usepackage{pgfplots}
\pgfplotsset{compat=1.14}
\usepackage{calc}

\author{Dor Elboim, Ofir Gorodetsky} 

\newcommand{\Addresses}{{
  \bigskip
  \footnotesize

  \textsc{Department of Mathematics, Princeton University, Princeton, NJ 08544, USA}\par\nopagebreak
  \textit{E-mail address:} \texttt{delboim@math.princeton.edu}

  \medskip

	\textsc{Mathematical Institute, University of Oxford, Oxford, OX2 6GG, UK}\par\nopagebreak
\textit{E-mail address:} \texttt{ofir.goro@gmail.com}
}}

\title{Uniform estimates for almost primes over finite fields}
\date{}
\usepackage[margin=1in]{geometry}
\usepackage[all]{xy}
\usepackage{amssymb}
\usepackage{amsfonts}
\usepackage{amsthm}
\usepackage{amsmath}
\usepackage{hyperref}
\usepackage{mathtools}
\usepackage{url}
\usepackage{xcolor}
\usepackage{pgfplots}

\newtheorem{thm}{Theorem}[section]

\newtheorem{lem}[thm]{Lemma}  
\newtheorem{prop}[thm]{Proposition}
\newtheorem{cor}[thm]{Corollary}

\newtheorem*{definition*}{Definition}
 
\newtheorem{remark}[thm]{Remark}

\newcommand{\PP}{\mathbb{P}}
\newcommand{\CC}{\mathbb{C}}
\newcommand{\ZZ}{\mathbb{Z}}
\newcommand{\FF}{\mathbb{F}}
\newcommand{\EE}{\mathbb{E}}

\newcommand{\re}{\Re}
\newcommand{\Mnq}{\mathcal{M}_{n,q}}
\newcommand{\Fr}{\mathrm{Fr}}
\newcommand{\Po}{\mathrm{Po}}
\newcommand{\Var}{\mathrm{Var}}

\mathtoolsset{showonlyrefs}

\numberwithin{equation}{section}

\begin{document}

\maketitle
\begin{abstract}
We establish a new asymptotic formula for the number of polynomials of degree $n$ with $k$ prime factors over a finite field $\FF_q$. The error term tends to $0$ uniformly in $n$ and in $q$. Previously, asymptotic formulas were known either for fixed $q$, through the works of Warlimont and Hwang, or for small $k$, through the work of Arratia, Barbour and Tavar\'e.

As an application, we estimate the total variation distance between the number of cycles in a random permutation on $n$ elements and the number of prime factors of a random polynomial of degree $n$ over $\FF_q$. The distance tends to $0$ at rate $1/(q\sqrt{\log n})$. Previously this was only understood when either $q$ is fixed and $n$ tends to $\infty$, or $n$ is fixed and $q$ tends to $\infty$, by results of Arratia, Barbour and Tavar\'{e}.
\end{abstract}
\section{Introduction}
Given a positive integer $n$, we let $\pi_n$ be a permutation chosen uniformly at random from $S_n$. Given a prime power $q$, we let $f_n=f_{n,q} \in \FF_q[T]$ be a polynomial chosen uniformly at random from  $\Mnq\subseteq \FF_q[T]$, the set of monic polynomials of degree $n$ over the finite field $\FF_q$.

We denote by $\Omega(f)$ the number of monic prime factors dividing a polynomial $f$, counted with multiplicity, and by $K(\pi)$ the number of cycles in a permutation $\pi$. We define the following function:
	\begin{equation}\label{eq:hq}
	h_q(x) :=  \prod_{P \in \mathcal{P}} \left( 1-\frac{x}{|P|}\right)^{-1} \left(1-\frac{1}{|P|}\right)^{x},
	\end{equation}
where $\mathcal{P}=\mathcal{P}_{q}$ is the set of monic irreducible polynomials over $\FF_q$ and $|f|=q^{\deg (f)}$. Note that $h_q(x)$ blows up when $x \to q^{-}$. Our main result, Theorem~\ref{thm:pointwise} below, compares $\PP(\Omega(f_n)=k)$ with $\PP(K(\pi_n)=k)$. Throughout the paper, $n \ge 2$, $1 \le k \le n$ and
\begin{equation}
r:= \frac{k-1}{\log n}.
\end{equation}
Unless stated otherwise, constants, both implied and explicit, are absolute. As Theorem~\ref{thm:pointwise} is somewhat technical, we first state two corollaries. As $n \to \infty$, both $K(\pi_n)$ and $\Omega(f_n)$ become concentrated around their mean, which is $\log n+O(1)$. The next corollary shows that the ratio of $\PP(\Omega (f_{n})=k)$ and $\PP(K (\pi_{n})=k )$ is asymptotic to $h_q(r)$, in the most general limit $q^n \to \infty$, for $k$ as large as $C \log n$ for an explicit $C>1$.
\begin{cor}\label{cor:pointwise}
	For $r \le 3/2$ we have
	\begin{equation}\label{eq:pointwise}
	\left| \frac{\PP \left(\Omega (f_{n})=k \right)}{\PP (K(\pi _n)=k)} -   h_q( r ) \right| \le \frac{Ck}{q(\log n)^2}, \qquad q^n \to \infty.
	\end{equation}
\end{cor}
As we shall see in Lemma~\ref{lem:h size}, $h_q(r)  \ge c$, and so \eqref{eq:pointwise} gives an asymptotic result.

Both $K(\pi_n)$ and $\Omega(f_n)$ are supported on $[n]:=\{1,2,\ldots,n\}$. Denote by $\mu_{K,n}$ and $\mu_{\Omega,n}$ the distributions of $K(\pi_n)$ and $\Omega(f_n)$, which are measures on this set. Another corollary of our main result is an estimate for the total variation distance of the two measures.
\begin{cor}\label{cor:tv}
	As $q^n$ tends to infinity, we have
	\begin{equation}
	d_{\mathrm{TV}}(\mu_{K,n}, \mu_{\Omega,n}) := \frac{1}{2} \sum_{k \in [n]} \left| \PP(K(\pi_n)=k) - \PP(\Omega(f_n)=k) \right| = \Theta\left( \frac{1}{q\sqrt{\log n}} \right).
	\end{equation}
\end{cor}
The main contribution to the total variation comes from values near $\log n$. As $h_q(1)=1$, if follows from Corollary~\ref{cor:pointwise} that $\PP \left(\Omega (f_{n})=k \right)$ and $\PP (K(\pi _n)=k)$ are close when $k$ is near $\log n$, which explains heuristically why the total variation tends to $0$ despite the correction factor $h_q(r)$. 

We now state the main result. Let $X=X_n$ be a Poisson random variable with mean $\log n$.
\begin{thm}\label{thm:pointwise}
	Fix $\delta \in (0,1)$. Suppose $n\ge 4(1-\delta)/\delta^2$ and $q\ge 1/(1-\delta)^{2}$. For $r \le q(1-\delta)$ we have
	\begin{equation}\label{eq:ineq}
	\left| \PP \left(\Omega (f_{n})=k \right) - \PP (K(\pi _n)=k)   h_q( r ) \right| \le C_{\delta}(r+1)^{C_{\delta}r} \PP(X=k-1)   \frac{k}{q(\log n)^2}.
	\end{equation}
\end{thm}
Our theorem reduces the asymptotic study of $\PP(\Omega(f_n)=k)$ to that of $\PP(K(\pi_n)=k)$, at least in a certain range (see Remark~\ref{rem:range} for a discussion of the range). By definition, $\PP(K(\pi_n)=k)=|s(n,k)|/n!$ where $s(n,k)$ are the Stirling numbers of the first kind. Asymptotics of these numbers were studied, in the entire range $1\le k\le n$, by Moser and Wyman \cite{moser1958}.
\begin{remark}\label{rem:range}
From the work of Moser and Wyman, one can show that $\PP(X=k-1) \le Ce^{Cr^2}\PP(K(\pi_n)=k)$,  so that Theorem~\ref{thm:pointwise} implies
\[
\left| \frac{\PP(\Omega(f_n)=k)}{\PP(K(\pi_n)=k)}- h_q(r) \right| \le C_{\delta}e^{C_{\delta}r^2}\frac{k}{q(\log n)^2}
\]
when $r \le q(1-\delta)$. Since $h_q(r) \ge 1$ for $r \ge 1$, it follows that we have an asymptotic result whenever $r \le c_{\delta} \sqrt{\log (q\log n)}$. However, we do not attempt to determine the widest range where $\PP(\Omega(f_n)=k)/\PP(K(\pi_n)=k) \sim h_q(r)$ holds, as the current result suffices for our corollaries.
\end{remark}

\subsection{Previous works on pointwise bounds}
Given a positive integer $n$, we denote by $\Omega(n)$ the number of its prime factors, counted with multiplicity. For a real number $x>1$, we denote by $N_x$ an integer chosen uniformly at random from $[1,x]\cap \ZZ$. Landau proved that \cite{landau1909}
\begin{equation}\label{eq:land}
 \PP(\Omega(N_x)=k)\sim \frac{1}{\log x} \frac{(\log \log x)^{k-1}}{(k-1)!}
\end{equation}
as $x \to \infty$, for any fixed $k \ge 1$. For $k=1$ this is the Prime Number Theorem. For $k$ growing with $x$, one has the following result, proved by Sathe \cite{sathe1}, whose proof was greatly simplified by Selberg \cite{selberg1954}. Fix $\delta \in (0,2)$. Uniformly for $x \ge 3$ and $1 \le k \le (2-\delta) \log \log x$, one has
\begin{equation}\label{eq:ss2}
\PP(\Omega(N_x)=k)= \frac{1}{\log x} \frac{(\log \log x)^{k-1}}{(k-1)!} \left(H\left( \frac{k-1}{\log \log x} \right) +O_{\delta}\left( \frac{k}{(\log \log x)^2} \right) \right)
\end{equation}
as $x \to \infty$, where
\begin{equation}
H(x):= \frac{1}{\Gamma(x+1)} \prod_{p \text{ prime}} \left( 1-\frac{x}{p}\right)^{-1} \left(1-\frac{1}{p}\right)^{x}.
\end{equation}
The proof is now a part of the general Selberg-Delange-Tenenbaum method, which is explained in detail in \cite[Ch.~II.5]{tenenbaum2015}.

Moser and Wyman \cite{moser1958} gave a simple asymptotic formula for $\PP(K(\pi_n)=k) = |s(n,k)|/n!$ in the range $k=o(\log n)$, and a more complicated one, involving some implicit constants, for the complimentary range. Since we are interested in the wider range $k=O(\log n)$, we state the following result of Hwang \cite{hwang1995}, proved by adapting the Selberg-Delange-Tenenbaum method:
\begin{equation}\label{eq:omega perm intro}
\PP\left( K(\pi_n) =k \right) = \frac{1}{n}  \frac{(\log n)^{k-1}}{(k-1)!} \frac{1}{\Gamma(r+1)} \left( 1+  O_{A}\left( \frac{k}{(\log n)^2} \right) \right)
\end{equation}
as $n \to \infty$, uniformly for $1 \le k \le A \log n$. 

For $n \to \infty$ and fixed $q$, Warlimont \cite{warlimont1993} proved that if we fix $\delta \in (0,q)$, then
\begin{equation}\label{eq:omega fnq intro}
\PP\left( \Omega(f_{n}) =k \right) = \frac{1}{n}\frac{(\log n)^{k-1}}{(k-1)!} \frac{1}{\Gamma(r+1)}\left(  h_q(r) + O_{\delta,q}\left( \frac{1}{\log n} \right) \right),
\end{equation}
uniformly for $1 \le k \le (q-\delta) \log n$. This is an analogue of \eqref{eq:ss2}; see also Car \cite{car1982} and Afshar and Porritt \cite{afshar2019}. Our Theorem~\ref{thm:pointwise} implies \eqref{eq:omega fnq intro} with the improved error term $k/(\log n)^2$. Indeed, for $n \to \infty$ and fixed $q$ and $\delta \in (0,1)$, we have $\PP(X_n=k-1) = O_{\delta,q}(\PP(K(\pi_n)=k))$ for $r \le q(1-\delta)$ by \eqref{eq:omega perm intro}, so that \eqref{eq:ineq} takes the form $\PP(\Omega(f_n)=k) = \PP(K(\pi_n)=k) (h_q(r) + O_{q,\delta}(k/(\log n)^2))$. By \eqref{eq:omega perm intro}, this implies \eqref{eq:omega fnq intro}.

In the opposite limit, where $q \to \infty$ while $1 \le k \le n$ are fixed, we have
\begin{equation}\label{eq:q lim}
\PP(\Omega(f_n)=k) = \PP(K(\pi_n)=k) \left( 1+O_n\left( \frac{1}{q}\right) \right)
\end{equation}
by a standard argument, see Remark~\ref{rem:q tv} below. We achieve an asymptotic formula for $\PP\left( \Omega(f_{n}) =k \right)$, which holds in the most general limit $q^n \to \infty$, by replacing the main term 
\begin{equation}\label{eq:old mt}
\frac{1}{n}\frac{(\log n)^{k-1}}{(k-1)!} \frac{h_q(r)}{\Gamma(r+1)},
\end{equation}
found by Warlimont, by a different one\footnote{See \cite{gorodetsky2017} for another example where modifying the main term leads to results in the $q^n \to \infty$ limit.}:
\begin{equation}\label{eq:new mt}
\PP(K(\pi_n)=k) h_q(r).
\end{equation}
These terms are asymptotic, in the large-$n$ limit, by the work of Hwang.

An uniform estimate for $\PP\left(\Omega(f_n)=k\right)$, in a limited range, was established previously by Arratia, Barbour and Tavar\'e \cite[Thm.~6.1]{arratia1993}, who proved that
\begin{equation}\label{eq:abr q}
\PP \left(\Omega (f_{n})=k \right) = 	\PP \left(K (\pi_{n})=k \right)\left( 1+ O\left( \frac{k}{q(\log n -k)} \right)\right), \qquad k<\log n,
\end{equation}
for $n>1$. Their proof is probabilistic and uses a coupling argument. Corollary~\ref{cor:pointwise} implies \eqref{eq:abr q}, since $h_q(r)=1+O(r/q)$ for $r \le 1$, by Lemma~\ref{lem:deriv2}.

A computation of Afshar and Porritt \cite[\S5]{afshar2019} shows that
\begin{equation}
\PP\left( \Omega(f_n)=k\right) = \PP \left(K (\pi_{n})=k \right)\left( 1+ O\left( \frac{kn}{q}\right)\right), \qquad kn=O(q).
\end{equation}
This gives an asymptotic estimate whenever $q$ grows faster than $kn$.  

Finally, we mention another work of Hwang \cite{hwang1998}, who studied $\PP(\Omega(f_n)=k)$ in the entire range of $k$, in the setting where $q$ is fixed.

\subsection{Previous works on total variation}
We may interpret $\mu_{K,n}$ and $\mu_{\Omega,n}$ as follows. Let $S_n^{\#}$ be the space of conjugacy classes in $S_n$. We have a natural map $X\colon S_n \to S_n^{\#}$, as well as the map $\Fr \colon \Mnq \to S_n^{\#}$ defined as follows: if $f \in \Mnq$ factors as $\prod_{i=1}^{d} P_i$, $\Fr(f)$ is the conjugacy class with cycle lengths $(\deg(P_i))_{i=1}^{d}$. For squarefree $f$, this map arises by labelling the roots of $f$ in the algebraic closure of $\FF_q$ and considering the permutation induced on them by the action of the Frobenius $x \mapsto x^q$. 
Letting $\mu_{S}$ be the uniform measure on a finite set $S$, we have two measures on $S_n^{\#}$: $\mu_{n}:=X_{*} \mu_{S_n}$ and $\mu_{n,q} := \Fr_{*} \mu_{\Mnq}$, where we use $A_{*}B$ to denote the pushforward of the measure $B$ under the map $A$. In this notation, $\mu_{K,n}= K_{*}\mu_n$ and $\mu_{\Omega,n} = K_{*}\mu_{n,q}$.  

The total variation distance of $\mu_{n,q}$ and $\mu_n$ was studied by Arratia, Barbour and Tavar\'{e} \cite[Cor.~5.6]{arratia1993}, who showed that it is of order $\Theta(1/q)$; see \cite{barysoroker2018} for an alternative proof by Bary-Soroker and the second author. This implies that
\begin{equation}\label{eq:tv q}
d_{\mathrm{TV}}(\mu_{K,n}, \mu_{\Omega,n}) = O\left( \frac{1}{q} \right).
\end{equation}
Additionally, in \cite[Thm.~6.8]{arratia1993} it is proved that 
\begin{equation}\label{eq:tv n}
d_{\mathrm{TV}}(\mu_{\Omega,n}, \Po(H_n))  = O\left( \frac{1}{\sqrt{\log n}}\right),
\end{equation}
where $H_n$ is the $n$th harmonic number and $\Po(\lambda)$ is the Poisson distribution with mean $\lambda$. From \eqref{eq:tv q} and \eqref{eq:tv n} and the triangle inequality, it follows by taking $q$ to infinity that \eqref{eq:tv n} holds with $\mu_{\Omega,n}$ replaced by $\mu_{K,n}$. An additional application of the triangle inequality yields
\begin{equation}\label{eq:tv in n}
d_{\mathrm{TV}}(\mu_{K,n}, \mu_{\Omega,n})  = O\left( \frac{1}{\sqrt{\log n}} \right).
\end{equation} 
Corollary~\ref{cor:tv} improves upon both \eqref{eq:tv q} and \eqref{eq:tv in n}, and is optimal.
\begin{remark}\label{rem:q tv}
From \eqref{eq:tv q}, $\PP(\Omega(f_n)=k) = \PP(K(\pi_n)=k) + O(1/q)$ and \eqref{eq:q lim} follows. In fact, the much weaker estimate $d_{\mathrm{TV}}(\mu_{K,n}, \mu_{\Omega,n}) = O_n(1/q)$ suffices; see \cite[Eq.~(2.3)]{cohen1970} or \cite[Lem.~2.1]{andrade2015} for a proof of it.
\end{remark}
\subsection*{Acknowledgments}
We thank Andrew Granville and G\'{e}rald Tenenbaum for feedback on an earlier version of the manuscript, and the anonymous referee for useful comments. OG was supported by the European Research Council (ERC) under the European Union’s 2020 research and innovation programme (ERC grant agreements nos 786758 and 851318).
\section{Preparation}
In what follows, $C$ and $c$ are always absolute constants whose values might change from one occurrence to the next. When constants appear with a subscript, their value may depend on the parameters in the subscript.
\subsection{Primes}
We denote by $\pi_q(n):=|\mathcal{P} \cap \Mnq|$ the number of primes of degree $n$. From Gauss's identity $\sum_{d \mid n} d \pi_q(d) =q^n$ \cite[Eq.~(1.3)]{arratia1993} we have the estimates
\begin{equation}\label{eq:ppt}
n\pi_q(n) \le q^n \text{ and }n\pi_q(n) = q^n + O(q^{\lfloor n/2\rfloor}), 
\end{equation}
which shall be used frequently.
\subsection{Generating functions}
We define the following power series:
\begin{equation}
\begin{split}
F(u,z) &= \sum _{n,k \ge 0}   \PP (K (\pi _n)=k) u ^n  z^k, \\
F_{q}(u,z) &= \sum _{n,k \ge 0}   \PP (\Omega (f_{n})=k) u ^n z^k .
\end{split}
\end{equation}
Since $\PP (K(\pi _n)=k)$ and $\PP (\Omega (f_{n})=k)$ are between $0$ and $1$, these series converge absolutely in 
\begin{equation}
A:=\{ (u,z) \in \CC \times \CC : |u|<1,\, |z| < 1\}
\end{equation}
and define analytic functions in that domain. We shall show that they can be analytically continued to a larger region. The logarithm function will always be used with its principal branch. Define the infinite product
\begin{equation}
H_q(u,z) := \prod _{P \in \mathcal{P}} \frac{\left( 1- \left( \frac{u}{q} \right)^{\deg (P)} \right)^z}{1-z \left( \frac{u}{q} \right)^{ \deg (P)}},
\end{equation}
so that $H_q (1, x )=h_q(x )$. Here $(1-(u/q)^{\deg(P)})^{z} = \exp(z\log(1-(u/q)^{\deg(P)}))$. In the next lemma we study the convergence of $H_q(u,z)$ in
\begin{equation}
B:=\{(u,z) \in \CC\times \CC  : |u|<\sqrt{q},\, |uz|<q \}.
\end{equation}
\begin{lem}\label{lem:hq convergence}
$H_q(u,z)$ converges uniformly to an analytic function on every compact subset of $B$. 
\end{lem}
\begin{proof}
For any $P \in \mathcal{P}$, let 
\begin{equation}\label{eq:def hP}
h_P(u,z) := \frac{\left( 1- \left( \frac{u}{q} \right)^{\deg (P)} \right)^z}{1-z \left( \frac{u}{q} \right)^{ \deg (P)}},
\end{equation}
which is analytic in $B$. We have
\begin{equation}\label{eq:sum gP}
\log h_P(u,z) =  \sum_{i \ge 2} \frac{(\frac{u}{q})^{\deg(P)i}}{i}(z^i - z) 
\end{equation}
in $B$. Fix a a real number $r\in(0,\sqrt{q})$, and consider the compact subset $B_{r} :=  \{ (u,z) \in \CC \times \CC : |u|\le r,\, |z|\le (\sqrt{q}-r)^{-1}, \,  |uz/q| \le r/\sqrt{q}\}$ of $B$. Any compact subset of $B$ is contained in $B_{r}$ for some $r$. 
We have, by the triangle inequality,
\begin{equation}\label{eq:log up n}
\begin{split}
\sum_{\deg(P) \le N} \left| \log h_P(u,z) \right| &\le \sum_{\deg(P) \le N} \sum_{i \ge 2} \frac{\big|\frac{u}{q}\big|^{\deg(P)i}}{i}(|z|^i + |z|)\\
&= \sum_{n \ge 1} \frac{\big|\frac{u}{q}\big|^n}{n} \sum_{\substack{d \le N\\ d\mid n,\, d \neq n}} d\pi_q(d) (|z|^{n/d} + |z|)
\end{split}
\end{equation}
for $(u,z) \in B_r$. Recall $\pi_q(d) \le q^d/d$. We may assume without loss of generality that $|z| \ge 1$ (by possibly increasing $r$), since the right-hand side of \eqref{eq:log up n} is increasing in $|z|$. The function $s(t)=q^{t}|z|^{n/t}$ on $[1,\min\{N,n/2\}]$ attains its maximum on one of the endpoints (since $(\log s(t))'' \ge 0$). Hence we have in $B_{r}$
\begin{equation}
\sum_{\deg(P) \le N} \left| \log h_P(u,z) \right|  \le \sum_{n \ge 1}  \big| \frac{u}{q} \big|^n  \max_{1 \le t \le \min\{N,n/2\}} (q^t |z|^{n/t}) +  \sum_{n \ge 1} (\frac{r}{q})^n q^{n/2}|z|   =: S_1 + S_2.
\end{equation}
We bound $S_1$:
\begin{equation}
\begin{split}
S_1 &\le \sum_{n \ge 1} \big|\frac{u}{q}\big|^n (q|z|^n + q^{\min\{N,n/2\}} |z|^{n/\min\{N,n/2\}}) \\
& =  q \sum_{n \ge 1} \big|\frac{uz}{q}\big|^n + |z|^2\sum_{n \le 2N} \big|\frac{u}{\sqrt{q}}\big|^n + q^{N}\sum_{n > 2N} \big( \frac{|u||z|^{1/N}}{q} \big)^n .
\end{split}
\end{equation}
The first sum is at most $q \sum_{n \ge 1} (r/\sqrt{q})^n = qr/(\sqrt{q}-r)$. The second sum is at most $|z|^2 \sum_{n \ge 1} (r/\sqrt{q})^n = |z|^2 r/(\sqrt{q}-r)$. If $|z|< 1$, the third sum is at most $q^N \sum_{n>2N} (1/\sqrt{q})^n \le 4/\sqrt{q}$. Otherwise, $|uz^{1/N}/q|\le|uz/q|<1$ and so the third sum is $q^N( |u||z|^{1/N}/q)^{2N+1}/(1-|u||z|^{1/N}/q)\le q^{-1}|z|^3|u| (|u|^2/q)^{N} /(1-r/\sqrt{q}) \le (r/(\sqrt{q}(\sqrt{q}-r)^4)) (r^2/q)^N$. We evaluate $S_2$:
\begin{equation}
S_2  =|z|\sum_{n \ge 1} (\frac{r}{\sqrt{q}})^n = |z| \frac{r}{\sqrt{q}-r}.
\end{equation}
All in all,
\begin{equation}
\sum_{\deg(P) \le N} \left|  \log h_P(u,z)\right| \le \frac{(q+|z|^2 + |z|)r}{\sqrt{q}-r} + \frac{4}{\sqrt{q}} + \frac{r}{\sqrt{q}(\sqrt{q}-r)^4}\left(\frac{r^2}{q} \right)^N
\end{equation}
for $(u,z) \in B_r$. Taking $N$ to infinity, we find that $\sum_{P \in \mathcal{P}}  |\log h_P(u,z)|$ converges and is bounded by a constant independent of $(u,z)\in B_r$. This proves that $H_q(u,z)$ converges uniformly to an analytic function on $B_{r}$. 
\end{proof}
\begin{lem}
For $(u,z) \in A$ we have
\begin{equation}
\begin{split}
F(u,z) &= (1-u)^{-z}, \\
F_{q}(u,z) &= (1-u)^{-z} H_q(u,z).
\end{split}
\end{equation}
\end{lem}
\begin{proof}
By the exponential formula for permutations \cite[Cor.~5.1.9]{stanley1999}, we have the equality
\begin{equation}\label{eq:expo}
\sum _{n \ge 0} \sum_{ \pi \in S_n} \frac{z^{K(\pi)}}{n!} u^n= \exp \left( \sum _{i=0}^{\infty } \frac{z}{i} u^i \right)=\exp \left( -z \log (1-u) \right)= (1-u)^{-z},
\end{equation}
which should be interpreted as equality of formal power series. The left-hand side of \eqref{eq:expo} is $F(u,z)$. Since both sides of \eqref{eq:expo} define analytic function in $A$, the uniqueness principle implies $F(u,z)=(1-u)^{-z}$ in $A$. We have
\begin{equation}
\prod _{P \in \mathcal{P}} \left( 1- \left(\frac{u}{q}\right)^{\deg (P)}\right)^{-1} = \prod _{P \in \mathcal{P}} \left(  \sum_{n=0}^{\infty} \left(\frac{u}{q} \right)^{ \deg (P^n)}  \right) =  \sum _{f \in \FF_q[T], \text{ monic}} \frac{u^{\deg (f) }}{q^{\deg (f)}} =  \sum _{n=0}^{\infty } u^n=(1-u)^{-1}
\end{equation}
for $|u|<1$. Hence
\begin{equation}
\begin{split}
(1-u)^{-z} H_q(u,z) &= \prod _{P \in \mathcal{P}} \left( 1- z\left( \frac{u}{q} \right)^{\deg (P)} \right)^{-1} = \prod_{P \in \mathcal{P}} \left( \sum_{n \ge 0} z^n \left( \frac{u}{q} \right)^{\deg (P^n)} \right).
\end{split}
\end{equation}
Fix a positive integer $N$. For real $u,z\in (0,1)$, we have, by unique factorization in $\FF_q[T]$,
\begin{multline}
\sum_{\substack{f \in \FF_q[T], \text{ monic}\\\deg(f) \le N}} \left( \frac{u}{q} \right) ^{\deg(f)} z^{\Omega(f)}  \le \prod_{P \in \mathcal{P},\, \deg(P)\le N} \left( 1+ z \left( \frac{u}{q} \right)^{\deg(P)}+z^2\left( \frac{u}{q} \right)^{\deg(P^2)} + \ldots + z^N \left( \frac{u}{q} \right)^{\deg(P^N)} \right)\\
 \le \prod_{\deg(P)\le N} \left( \sum_{n \ge 0} z^n \left( \frac{u}{q} \right)^{\deg (P^n)} \right) \le \prod_{P \in \mathcal{P}}\left( \sum_{n \ge 0} z^n \left( \frac{u}{q} \right)^{\deg (P^n)} \right) = (1-u)^{-z} H_q(u,z).
\end{multline}
Letting $N \to \infty$, we obtain $F_{q}(u,z) \le (1-u)^{-z} H_q(u,z)$. To prove the reverse inequality, fix positive integers $N<M$ and note that, again by unique factorization,
\begin{equation}
\prod_{\deg(P)\le N}\left( 1 + z\left( \frac{u}{q} \right)^{\deg(P)}+z^2\left( \frac{u}{q} \right)^{\deg(P^2)} + \ldots + z^M\left( \frac{u}{q} \right)^{\deg(P^M)} \right) \le \sum_{f \in \FF_q[T], \text{ monic}} \left( \frac{u}{q} \right)^{\deg(f)} z^{\Omega(f)}.
\end{equation}
Letting $M \to \infty$ we obtain $\prod _{\deg(P)\le N} \left( \sum_{n \ge 0} z^n \left( u/q \right)^{\deg (P^n)} \right) \le F_{q}(u,z)$. Letting $N \to \infty$ we obtain $(1-u)^{-z} H_q(u,z) \le F_{q}(u,z)$. Thus $(1-u)^{-z} H_q(u,z)$ and $F_{q}(u,z)$ agree on $(0,1)\times(0,1)$ and so by the uniqueness principle are equal.
\end{proof}
From now on we consider the function $(1-u)^{-z}$ as an analytic function in $\CC \times (\CC \setminus [1,\infty))$, by using the definition $(1-u)^{-z} = \exp( -z\log (1-u))$.
\begin{lem}\label{lem:deriv2}
Fix $\delta \in (0,1)$.	Suppose $q\ge (1-\delta)^{-2}$, $|u_0| \le (1-\delta)^{-1/2}$ and $|z_0| \le (1-\delta)q$.  Then 
	\begin{equation}
	\left| \big(\frac{\partial}{\partial u} H_q\big)(u_0,z_0) \right|, \, 	\left| \big(\frac{\partial}{\partial z} H_q\big)(u_0,z_0) \right|,\,
	\left| \big(\frac{\partial^2}{\partial z^2} H_q\big)(u_0,z_0) \right|\le C_{\delta}   \frac{|z_0|^2+1}{q}\exp\left( C_{\delta} \frac{|z_0|^2}{q} \right).
	\end{equation}
\end{lem}
\begin{proof}
	We have
	\begin{equation}
	H_q(u,z) = \exp\left( \log H_q(u,z)\right) = \exp\left( \sum_{n \ge 1} \frac{(\frac{u}{q})^n}{n} \sum_{d \mid n,\, d \neq n} d \pi_q(d) (z^{n/d}-z) \right),
	\end{equation}
	where the sum converges absolutely and uniformly in some neighborhood of $(u_0,z_0)$ by Lemma~\ref{lem:hq convergence} and its proof. For all $i,j \ge 0$,
		\begin{equation}\label{eq:partial log}
	(\frac{\partial^{i+j}}{\partial^i u \partial^j z} \log H_q)(u,z) = \sum_{n \ge 2} u^{n-i} q^{-n}\frac{n(n-1)\cdots (n-(i-1))}{n} \sum_{d \mid n,\, d \neq n} d \pi_q(d) (z^{\frac{n}{d}-j}\frac{n}{d}(\frac{n}{d}-1)\cdots (\frac{n}{d}-(j-1))-z^{1-j}),
	\end{equation}
	where $z^k$ should be interpreted as $0$ for negative $k$. Recall the bound $\pi_q(d) \le q^d/d$, and that the function $s(t)=q^{t}|z_0|^{n/t}$ on $[1,n/2]$ attains its maximum on one of the endpoints if $|z_0| \ge 1$. Otherwise, $s(t) \le q^{n/2}$. Hence
	\begin{equation}\label{eq:partial log2}
	\left| (\frac{\partial^{i+j}}{\partial^i u \partial^j z} \log H_q)(u_0,z_0) \right| \le C \sum_{n \ge 2} (1-\delta)^{-n/2}q^{-n}  n^{i+j}
	(q |z_0|^n + q^{n/2}(1+|z_0|^2))
	\end{equation}
	for all $i,j \ge 0$. As $\sum_{n \ge k} x^n n^m \le C_{k+m} x^k/(1-x)^{m+1}$ for $x \in (0,1)$, we find
	\begin{equation}\label{eq:partial log3}
	\left| (\frac{\partial^{i+j}}{\partial^i u \partial^j z} \log H_q)(u_0,z_0) \right| \le \frac{C_{i+j,\delta}(| z_0|^2+1)}{q} .
	\end{equation}
Since $(\exp(g))' = g' \exp(g)$ and $(\exp(g))'' = (g''+g'^2) \exp(g)$ for any analytic function $g$, we are done.
\end{proof}

\begin{lem}\label{lem:h size}
If $q>x \ge 1$,
\begin{equation}\label{eq:h lower 1}
h_q(x) \ge 1 + \frac{x-1}{2q} \ge 1.
\end{equation}
If $0\le x \le 1$,
\begin{equation}\label{eq:h lower 2}
h_q(x) \ge c.
\end{equation}
\end{lem}
\begin{proof}
By Bernoulli's inequality, $(1-1/|P|)^x \ge 1-x/|P|$ for $x \ge 1$, and so $h_q(x) \ge 1$ for $x \ge 1$. By considering the contribution of linear primes to $h_q(x)$ in \eqref{eq:hq}, we see that for $x \ge 1$,
\begin{equation}
h_q(x) \ge \left( 1-\frac{1}{q}\right)^{xq}\left(1-\frac{x}{q}\right)^{-q} = \exp\left( \sum_{i \ge 2} \frac{x^i-x}{iq^{i-1}} \right) \ge \exp\left( \frac{x^2-x}{2q} \right) \ge 1 + \frac{x^2-x}{2q} \ge 1 +\frac{x-1}{2q}.
\end{equation}
For $0 \le x \le 1$, we have $\log h_q(x) = O(x/q)$ by \eqref{eq:partial log3}, so that $h_q(x) \ge \exp(-cx/q) \ge c$.
\end{proof}

\subsection{Poisson distribution}
\begin{lem}\cite[Thm.~5.4]{mitzenmacher2005}\label{lem:poiss}
Let $X$ be a Poisson random variable with mean $\lambda >0$. We have $\PP(X \ge x) \le (e\lambda/x)^x e^{-\lambda}$ for $x>\lambda$.
\end{lem}
\subsection{Integral estimates}
Recall $1/(z\Gamma(z))$ is an entire function.
\begin{lem}\label{lem:gamma}
Let $G(z) = 1/(z\Gamma(z))$. We have $|G'(z)|,|G(z)| \le C(A+1)^{CA}$ for $|z| \le A$.
\end{lem}
\begin{proof}
The bound for $G$ is \cite[Ch.~6, Thm.~1.6]{stein2003} and the bound for $G'$ follows from the one for $G$ by Cauchy's integral formula.
\end{proof}
\begin{lem}\label{lem:binom and gamma}
	Fix $A>0$. For all $|z|\le A$ and $n \ge 1$ we have
	\begin{equation}\label{eq:binom}
	\left| \binom{n+z-1}{n} - \frac{n^{z-1}}{\Gamma(z)} \right|  \le C (A+1)^{CA}  n^{\Re z-2}.
	\end{equation}
\end{lem}
\begin{proof}
For $z$ a non-positive integer, the left-hand side is $0$ or sufficiently small. Otherwise, dividing by $n^{\Re z-2}$, it suffices to bound
\begin{equation}
\left|\frac{\Gamma(n+z)}{\Gamma(n+1)\Gamma(z)n^{z-2}} - \frac{n}{\Gamma(z)} \right|,
\end{equation}
where $|z| \le A$ and $z \neq 0,-1,\ldots$. If $n\ge 2A+1$, $\re (n+z) \ge n/2$ and we may apply Stirling's approximation to find $\Gamma(n+z)/(\Gamma(n+1)n^{z-2}) = n+O((A+1)^{CA})$ and the desired bound follows from Lemma~\ref{lem:gamma}. If $n < 2A+1$, the terms $n$, $|1/n^{z-2}|$ and $|\Gamma(n+z)/\Gamma(z)|=|(n+z-1)(n+z-2)\cdots (z)|$ are all bounded from above by $O((A+1)^{CA})$, as well as $|1/\Gamma(z)|$, $1/\Gamma(n+1)$ by Lemma~\ref{lem:gamma}, which finishes the proof.
\end{proof}
For the rest of this section, let $X=X_n$ be a Poisson random variable with mean $\log n$.
\begin{lem}\label{lem:3integrals}
Let $n \ge k>1$ and set $r=(k-1)/\log n$. Let $\beta$ be the circle $|z|=r$ oriented counterclockwise. For $j \ge 0$ we have
\begin{align}\label{eq:int 1}
\int_{\beta} \left| \frac{(z-r)^j n^{z-1}}{z^{k}}\right| \, \left| dz\right|&\le C_j \PP( X=k-1) \left( \frac{\sqrt{k}}{\log n}\right)^j,\\
\label{eq:int gamma}\int_{\beta} \left| \frac{(z-r)^j n^{z-1}}{z^{k+1}\Gamma(z)}\right| \, \left| dz\right|&\le C_j\PP( X=k-1) \left( \frac{\sqrt{k}}{\log n}\right)^j (r+1)^{Cr},
\end{align}
\begin{equation}\label{eq:int 2}
\int_{\beta}\frac{(z-r) n^{z-1}}{z^{k}} \,  dz=0.
\end{equation}
\end{lem}
\begin{proof}
Using the parametrization $z=re^{it}$ and the estimate $\cos t -1 \le -ct^2$ for $t \in [-\pi,\pi]$,
\begin{equation}
\int_{\beta} \left| \frac{(z-r)^j n^{z-1}}{z^{k}}\right| \, \left| dz\right|\le \frac{n^{r-1}r^j}{r^{k-1}} \int_{-\pi}^{\pi} \left| e^{it} -1 \right|^j n^{-rct^2}\, dt \le \frac{n^{r-1}r^j}{r^{k-1}} \int_{-\pi}^{\pi}\left|t\right|^j n^{-rct^2}\, dt,
\end{equation}
and we conclude \eqref{eq:int 1} by using the change of variables $(k-1)t^2=s^2$ and Stirling's approximation. To obtain \eqref{eq:int gamma} we repeat the computation and appeal to Lemma~\ref{lem:gamma}. To obtain \eqref{eq:int 2}, observe that the coefficient of $z^{k-1}$ in $(z-r)n^{z-1}$ is 
\begin{equation}
n^{-1} \left( \frac{(\log n)^{k-2}}{(k-2)!} - r \frac{(\log n)^{k-1}}{(k-1)!} \right) =0, 
\end{equation}
as needed.
\end{proof}
\begin{figure}[!tbp]
	\hspace*{\fill}
	\begin{minipage}[b]{0.45\textwidth}
		\begin{tikzpicture}
			\begin{axis}[ticks=none,axis lines=middle,xmin=-2,xmax=2,ymin=-2,ymax=2,axis equal image]
				\draw [dash pattern=on 1pt off 3pt, domain=-180:180] plot({cos(\x)},{sin(\x)});
				\addplot[domain=1:1.495]{-0.15};
				\addplot[domain=1:1.495]{0.15};
				\draw [samples =100,domain=5.7:354.3] plot({1.5*cos(\x)},{1.5*sin(\x)});
				\draw [domain=90:270] plot({1+0.15*cos(\x)},{0.15*sin(\x)});
				\node[text width=0cm] at (0.9,0.3)
				{$\gamma _1 $};
				\node[text width=0cm] at (1,1.2)
				{$\gamma _2 $};
				\draw[->,thick] (-1.059,1.059) -- (-1.059-0.001,1.059-0.001) node [pos=0.66,above]{} ;
				\draw[->,thick] (-1.059,-1.059) -- (-1.059+0.001,-1.059-0.001) node [pos=0.66,above]{} ;
				\draw[->,thick] (1.059,1.059) -- (1.059-0.001,1.059+0.001) node [pos=0.66,above]{} ;
				\draw[->,thick] (1.059,-1.059) -- (1.059+0.001,-1.059+0.001) node [pos=0.66,above]{} ;
				\draw[->,thick] (1.059,-1.059) -- (1.059+0.001,-1.059+0.001) node [pos=0.66,above]{} ;
			\end{axis}
		\end{tikzpicture}
	\end{minipage}
	\hspace*{\fill}
	\begin{minipage}[b]{0.45\textwidth}
		\begin{tikzpicture}
			\begin{axis}[ticks=none,axis lines=middle,xmin=-2,xmax=2,ymin=-2,ymax=2,axis equal image]
				\addplot[domain=0:1.3]{0.5};
				\addplot[dash pattern=on 1pt off 2pt,domain=1.3:1.8]{0.5};
				\addplot[dash pattern=on 1pt off 2pt,domain=1.3:1.8]{-0.5};
				\addplot[domain=0:1.3]{-0.5};
				\draw[samples =70, domain=90:270] plot({0.5*cos(\x)},{0.5*sin(\x)});
				\node[text width=0cm] at (1,0.6)
				{$\gamma_3 $};
				\draw[->,thick] (0.5,-0.5) -- (0.5-0.001,-0.5) node [pos=0.66,above]{} ;
				\draw[->,thick] (0.5-0.001,0.5) -- (0.5,0.5) node [pos=0.66,above]{} ;
			\end{axis}
		\end{tikzpicture}
	\end{minipage}
	\caption{The contour $\gamma= \gamma _1+\gamma _2$ in the $u$-plane and the contour $\gamma _3$ in the $v$-plane.}
	\label{fig:gammas figure}
\end{figure}
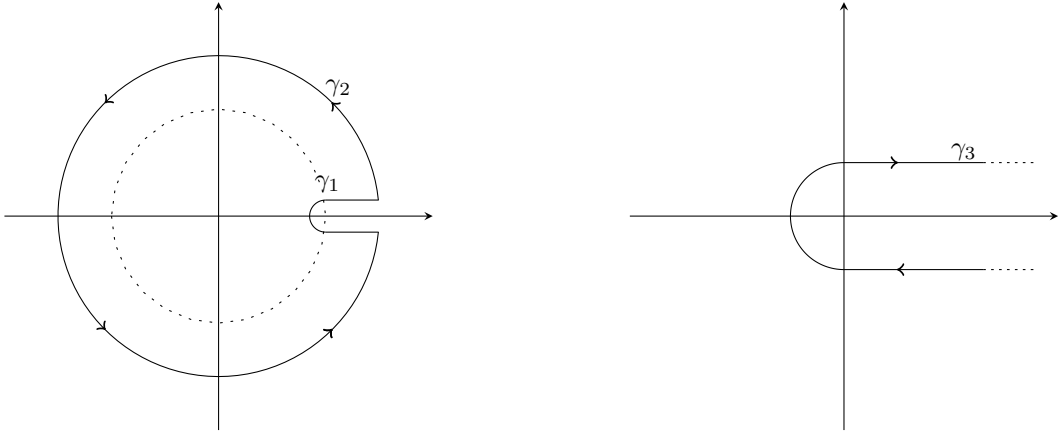

\begin{prop}\label{prop:I integral}
Let $n \ge k > 1$. Let $\beta$ be the circle $|z|=r$ oriented counterclockwise in the $z$-plane. Let $\gamma$ be the path in the $u$-plane depicted in Figure~\ref{fig:gammas figure}. In formulas, $\gamma$ is oriented counterclockwise as well, and we write it as a union of two curves, $\gamma_1$ and $\gamma_2$. Let $R=1+1/\sqrt{n}$ and define $\theta_1 \in (0,\pi)$ by $R \sin (\theta_1) = 1/n$. The curve $\gamma_1$ is $\gamma_1'+\gamma_1''+\gamma_1'''$, with \begin{equation}
\begin{split}
\gamma_1'(t) &= -\frac{i}{n}-t , \qquad t \in [-R\cos(\theta_1),-1],\\
\gamma_1''(\theta) &= 1+\frac{e^{i(2\pi-\theta)}}{n}, \qquad \theta \in [\pi/2, 3\pi/2],\\
\gamma_1'''(t) &= \frac{i}{n}+t , \qquad t \in [1,R\cos(\theta_1)],
\end{split}
\end{equation}
and $\gamma_2$ given by
\begin{equation}
\gamma_2(\theta) = R e^{i\theta}, \qquad \theta \in [\theta_1, 2\pi - \theta_1].
\end{equation} 	 
We have
	\begin{equation}\label{eq:int knq}
	\intop _{\beta } \intop _{\gamma } \frac{\left| (1-u)^{-z}\right|}{|u|^{n+1}|z|^{k+1}}  |u-1|  |du||dz| \le C\PP( X=k-1)(r+1)^{Cr}\frac{\log n}{nk}.
	\end{equation}
\end{prop}
\begin{proof}
	Let $I_1$ and $I_2$ be the integrals over $\beta \times \gamma _1$ and $\beta \times \gamma _2$, respectively:
	\begin{equation}
	I_i := \intop _{\beta } \intop _{\gamma_i } \frac{\left| (1-u)^{-z}\right|}{|u|^{n+1}|z|^{k+1}}  |u-1|  |du||dz|, \qquad i=1,2.
	\end{equation}
	By performing the change of variables $u = 1+n^{-1}v $, we obtain
	\begin{equation}\label{eq:i1 gamma3}
	I_1= \frac{1}{n^2} \intop _{\beta } \frac{n^{\re z}}{|z|^{k+1}} \intop_{\gamma_3  } \frac{\left| (-v) ^{-z} \right| |v|}{|1+n^{-1}v|^{n+1}} |dv||dz|,
	\end{equation}
	where $\gamma_3$ is depicted in Figure~\ref{fig:gammas figure}. We continue by bounding the inner integral:
	\begin{equation}
\begin{split}
	\max_{|z|\le r}\int_{\gamma_3}  \left| (-v)^{-z}  \right| \left|v\right| \left|1+n^{-1}v\right|^{-(n+1)}\, |dv| &\le e^{\pi r} \max_{|z|\le r}\int_{\gamma_3}  \left| v \right|^{C(r+1)} |1+n^{-1}v|^{-(n+1)}\, |dv|  \\
	& \le e^{\pi r} \big( C + C\int_{0}^{\infty}  t^{C(r+1)}e^{-ct}  \, dt \big)  \\
	&\le e^{\pi r}\Gamma(C(r+1)) \le C (r+1)^{Cr}.
	\end{split}
	\end{equation}
	We substitute the last bound in \eqref{eq:i1 gamma3}, parametrize $\beta$ as $z=re^{it}$ and use the inequality $\re z \le r(1-ct^2)$, which leads to 
	\begin{equation}
	I_1\le \frac{ C(r+1)^{Cr}}{n^2 r^k} \intop_{-\pi }^{\pi } n^{r(1-ct^2)} \, dt= \frac{C(r+1)^{Cr} n^{r-2}}{ r^k \sqrt{r\log n}} \intop_{-\pi \sqrt{r \log n}}^{\pi \sqrt{r \log n}} e^{-cs^2}\,ds  \le \frac{C(r+1)^{Cr}  n^{r-2}}{ r^k \sqrt{r\log n}}.
	\end{equation}
	Thus, by (a weak version of) Stirling's approximation we obtain
	\begin{equation}
	I_1 \le C(r+1)^{Cr} \PP( X=k-1)\frac{\log n}{nk}.
	\end{equation} 
	
	We turn to bound $I_2$. On $\beta \times \gamma_2$ we have $\left|(1-u)^{-z} \right|\le C\exp(\pi r + k/2)$, and so 
	\[I_2 \le C \frac{\exp(\pi r + \frac{k}{2})}{R^n r^k} \le   C(r+1)^{Cr}\PP(X=k-1)\exp(-ck-c\sqrt{n}),\]
	where here we again apply Stirling. As both $I_1$ and $I_2$ are bounded by the right-hand side of \eqref{eq:int knq}, we conclude the proof. 
\end{proof}

\section{Proof of Theorem~\ref{thm:pointwise}}
For $k=1$, the result follows from \eqref{eq:ppt}, so we may suppose $k>1$. 
Fix $\delta \in (0,1)$ and suppose $r \le q(1-\delta)$, $q\ge(1-\delta)^{-2}$ and $n\ge 4(1-\delta)/\delta^2$ (so that $1+1/\sqrt{n} \le (1-\delta)^{-1/2}$). By Cauchy's integral formula, we have 
	\begin{equation}
	\begin{split}
	\PP \left(K (\pi_n)=k \right)&= \Big(\frac{1}{2\pi i } \Big)^2 \intop _{ \beta } \intop _{\gamma } \frac{(1-u)^{-z}}{u^{n+1} z^{k+1}} \, du\, dz,\\
	\PP \left(\Omega (f _n)=k \right)&=\Big(\frac{1}{2\pi i } \Big)^2 \intop _{ \beta } \intop _{\gamma } \frac{(1-u)^{-z}}{u^{n+1} z^{k+1}} H_q (u,z)\, du\,dz,
	\end{split}
	\end{equation}
where $\beta$ and $\gamma$ are as defined in Proposition~\ref{prop:I integral}. Recall that $h_q(\bullet )= H_q (1, \bullet )$. Thus,
	\begin{equation}\label{eq:diff}
\PP \left(\Omega (f_n)=k \right)-\mathbb P (K (\pi _n)=k) h_q ( r)=  \Big(\frac{1}{2\pi i } \Big)^2 \intop _{ \beta } \intop _{\gamma } \frac{(1-u)^{-z}}{u^{n+1} z^{k+1}} (H_q (u,z)-H_q(1,r))\, du\,dz.
	\end{equation}
We have 
	\begin{multline}
	H_q (u,z)- H_q(1, r ) = (H_q(u,z) - H_q(1,z)) + (H_q(1,z) - H_q(1,r)) = O_{r,\delta,q}(|u-1|) + H_q(1,z) - H_q(1,r), 
	\end{multline}
	where the implied constant is, by Lemma~\ref{lem:deriv2},
	\begin{equation}\label{eq:imp}
	C_{\delta} \frac{r^2+1}{q} \exp\left( C_{\delta} \frac{r^2}{q} \right) \le C_{\delta} \frac{r^2+1}{q} \exp\left( C_{\delta} r\right).
	\end{equation}
Proposition~\ref{prop:I integral} shows that the total contribution of the $O_{r,\delta,q}(|u-1|)$-term to the right-hand side of \eqref{eq:diff} is acceptable. Since the $n$th coefficient of $(1-u)^{-z}$ is $\binom{n+z-1}{n}$, we can reduce to problem to a problem in the $z$-plane, namely bounding
	\begin{equation}
\intop _{ \beta } \intop _{\gamma } \frac{(1-u)^{-z}}{u^{n+1} z^{k+1}} (H_q (1,z)-H_q(1,r))\, du\,dz = 	\int_{\beta} \frac{\binom{n+z-1}{n} (H_q(1,z)-H_q(1,r))}{z^{k+1}}\, dz .
	\end{equation}
By Lemmas~\ref{lem:deriv2}, \ref{lem:binom and gamma} and \ref{lem:3integrals}, we may replace $\binom{n+z-1}{n}$ with $n^{z-1}/\Gamma(z)$ and $H_q(1,z)-H_q(1,r)$ with $(z-r) (\frac{\partial}{\partial z}H_q)(1,r) + O_{r,\delta,q}((z-r)^2)$ (the implied constant being again \eqref{eq:imp}), and the error terms will be acceptable. To bound the remaining integral, we use a first-order Taylor approximation for $G(z) = 1/(z\Gamma(z))$ to write
	\begin{equation}
\int_{\beta} \frac{n^{z-1} (z-r)}{\Gamma(z) z^{k+1}}\, dz = \frac{1}{\Gamma(r)r} \int_{\beta} \frac{n^{z-1} (z-r)}{z^k}\, dz + O\left( \max_{|t| \le  r}| G'(t)|\int_{\beta} \left| \frac{n^{z-1} (z-r)^{2}}{z^k}\right| \, \left|dz\right| \right).
\end{equation}
The main term vanishes by \eqref{eq:int 2}, and the error term is small enough by \eqref{eq:int 1} and Lemma~\ref{lem:gamma}. This finishes the proof. \qed
\section{Proof of Corollary~\ref{cor:pointwise}}
For $n \le 100$, the result follows from \eqref{eq:q lim} since $h_q(r) = 1+O(1/q)$ for $r \le 3/2$ by Lemma~\ref{lem:deriv2}. Otherwise, let us take $\delta = 1/5$ in Theorem~\ref{thm:pointwise} and obtain
\begin{equation}
\PP \left(\Omega (f_{n})=k \right) - \PP (K(\pi _n)=k) h_q(r) = O\left (\frac{\PP(X=k-1)k}{q(\log n)^2}\right)
\end{equation}
for all $n \ge 100$ and $q \ge 2$. The proof is finished by noting that $\PP(X=k-1) = O\left( \PP(K(\pi_n)=k)\right)$ uniformly in the range $k \le 3\log n/2$ by \eqref{eq:omega perm intro}. \qed
\section{Proof of Corollary~\ref{cor:tv}}
We may assume $n \ge C$, since for any fixed $n$ the following argument works. An upper bound of $O_n(1/q)$ on the total variation follows from Remark~\ref{rem:q tv}, while a lower bound of order $1/q$ follows from considering the contribution of $k=n$: 
\begin{equation}
\left| \PP(\Omega(f_n)=n) - \PP(K(\pi_n)=n) \right| = \frac{\binom{q+n-1}{n}}{q^n} - \frac{1}{n!} = \frac{1}{n!} \left( \prod_{i=1}^{n-1} \left( 1+ \frac{i}{q} \right) - 1 \right) \ge \frac{1}{q} \frac{1}{n!} \binom{n}{2}.
\end{equation}
Let $I_1 = [1, 3\log n/2]$, $I_2 = (3\log n/2, \sqrt{q} \log n]$, $I_3 = (\sqrt{q}\log n,n]$. For $1 \le i \le 3$, let $S_i$ be the contribution of $k \in I_i$ to the total variation:
\begin{equation}
S_i = \sum_{k \in I_i} \left| \PP(\Omega(f_n)=k) - \PP(K(\pi_n)=k) \right|.
\end{equation}
We shall show that $S_i = O(1/(q\sqrt{\log n}))$ for each $i$. Observe that $1=h_q(1)$ and that $h_q'(z) = O(1/q)$ for $|z| \le 3/2$ by Lemma~\ref{lem:deriv2}. By Theorem~\ref{thm:pointwise} and the estimate $h_q(z)-h_q(1) = O((z-1)/q)$,
\begin{equation}
\begin{split}
S_1 &= \sum_{k \in I_1} \left| \PP(K(\pi_n)=k) \left( h_q(r) - h_q(1) \right)+ O\left( \frac{\PP(X=k-1)k}{q(\log n)^2} \right) \right|\\
& \le \frac{C}{q}\left( \sum_{k \in I_1}  \PP(K(\pi_n)=k) \left| r -1 \right| + \sum_{k \in I_1}  \frac{\PP(X=k-1)k}{(\log n)^2} \right).
\end{split}
\end{equation}
From \eqref{eq:omega perm intro} we deduce the upper bound $\PP( K(\pi_n)=k) \le C  \PP( X = k-1)$ for $k \le 3 \log n/2$, so that
\begin{equation}
S_1 \le \frac{C}{q} \sum_{k \in I_1}  \PP(X=k-1) \left( \left| \frac{k-1}{\log n} -1 \right| +  \frac{k}{(\log n)^2} \right) \le \frac{C}{q} \left( \frac{\EE |X - \log n |}{\log n} + \frac{\EE X  + 1}{(\log n)^2} \right) \le \frac{C}{q\sqrt{\log n}},
\end{equation}
where the last inequality uses Cauchy-Schwarz: $\EE |X - \log n | \le \Var(X)^{1/2} = \sqrt{\log n}$. For $k\in I_2$, we have $h_q(r)-1 = O(r^3/q)$ by Lemma~\ref{lem:deriv2} with $\delta = 1/5$. By Theorem~\ref{thm:pointwise} with $\delta=1/5$,
\begin{equation}\label{eq:s2 bnd}
S_2 \le \frac{C}{q}\left( \sum_{k \in I_2}  \PP(K(\pi_n)=k) r^3 + \sum_{k \in I_2}  \frac{\PP(X=k-1)k}{(\log n)^2} (r+1)^{Cr} \right).
\end{equation}
We bound the first sum using Cauchy-Schwarz:
\begin{equation}
\sum_{k \in I_2} \PP(K(\pi_n)=k)r^3 \le \frac{\EE K^3(\pi_n) \cdot \mathbf{1}_{K(\pi_n)>3\log n/2}}{(\log n)^3}  \le \frac{\sqrt{\EE K^6(\pi_n) \PP( K(\pi_n) >3 \log n/2) }}{(\log n)^3} .
\end{equation}
By Markov's inequality and $\EE2^{K(\pi_n)} = n+1$ \cite[Thm.~13.3]{vanlint2001}, we have 
\[\PP( K(\pi_n) > 3 \log n/2) =\PP( 2^{K(\pi_n)} > n^{(\log 8)/2}) \le n^{-(\log 8)/2} \EE 2^{K(\pi_n)} = (n+1)n^{-(\log 8)/2}\le n^{-c}.\]
A similar argument shows $\PP( K(\pi_n)>10 \log n) = O(1/n^6)$, yielding $\EE K^6(\pi_n) \le C (\log n)^6$. Hence, the first sum in \eqref{eq:s2 bnd} is $O(n^{-c})$. To bound the second sum, we partition $I_2$ into intervals of length $\log n/2$:
\begin{equation}\label{eq:blocks}
\begin{split}
\sum_{k \in I_2}  \frac{\PP(X=k-1)k}{(\log n)^2} (r+1)^{Cr} &\le \sum_{j=3}^{\lfloor 2\sqrt{q} \rfloor} \sum_{\substack{k:\\ \frac{2k}{\log n} \in (j, j+1]}}\frac{\PP(X=k-1)k}{(\log n)^2}(r+1)^{Cr}\\
& \le \sum_{j\ge 3} \frac{\PP(X \ge \frac{j}{2} \log n-1) }{\log n} (j+1)^{Cj}.
\end{split}
\end{equation}
By Lemma~\ref{lem:poiss}, the probability in the right-hand side of \eqref{eq:blocks} is bounded by
\begin{equation}
\PP(X \ge \frac{j}{2}\log n-1) \le n^{\frac{j}{2}(1-\log\frac{j}{2})-1}e^{Cj} \le (j+1)^{-cj\log n} n^{-c} e^{Cj},
\end{equation}
where in the last inequality we use the fact that $(j/2)(1-\log(j/2)) -1$ is negative for all $j \ge 3$. Hence,
\begin{equation}
\sum_{k \in I_2}  \frac{\PP(X=k-1)k}{(\log n)^2} (r+1)^{Cr} \le n^{-c}\sum_{j \ge 3} (j+1)^{-cj \log n} (j+1)^{Cj} \le n^{-c} \end{equation}
for sufficiently large $n$. Substituting this bound into \eqref{eq:s2 bnd} we conclude that $S_2 \le 1/(qn^c)$. 

To bound $S_3$, recall that $\Var(K(\pi_n))=\log n + O(1)$ \cite{gontcharoff1942} and that $\Var(\Omega(f_n)) = \log n+O(1)$ (this is a function-field version of the main result of \cite{turan1934}), and both implied constants are absolute. Applying Chebyshev's inequality, we find $\PP( K(\pi_n) \ge \sqrt{q} \log n), \, \PP( \Omega(f_n) \ge \sqrt{q} \log n) \le C/(q \log n)$, and so $S_3 =O(1/(q \log n))$.

We now turn to prove a matching lower bound. Recall we may assume $n \ge C$. We consider the contribution to the total variation coming from $k-\log n \in [1, \sqrt{\log n}]$, which, by Corollary~\ref{cor:pointwise}, is
\begin{equation}\label{eq:sum c}
\sum_{k- \log n \in (0, \sqrt{\log n})}  \PP(K(\pi_n)=k)\left| h_q(r) - 1  \right| + O\left( \frac{1}{q \log n} \right).
\end{equation}
By \eqref{eq:omega perm intro}, $\PP(K(\pi_n)=k) \ge c \PP(X=k-1)$ for $r \le 3/2$. Additionally, $h_q(r) \ge 1+(r-1)/(2q)$ for $r \ge 1$ by \eqref{eq:h lower 2}. Hence, the last sum is bounded from below by
\begin{equation}
\frac{c}{q\log n}\sum_{k-\log n \in [1,\sqrt{\log n}]} \PP(X=k-1) \left| k-1 - \log n\right|.
\end{equation}
By Stirling's approximation, $\PP(X=i+\lfloor \EE X \rfloor ) \ge c/\sqrt{\log n}$ for $i=O(\sqrt{\log n})$, so that the last expression is bounded from below by 
\begin{equation}
\frac{c}{q(\log n)^{3/2}}  \sum_{3 \le i \le \sqrt{\log n}-3} i \ge \frac{c}{q \sqrt{ \log n}}.
\end{equation}
If $n$ is large enough, the error term in \eqref{eq:sum c} is small compared to $c/(q\sqrt{\log n})$, and the lower bound for the total variation follows. \qed 
\bibliographystyle{alpha}
\bibliography{references}

\begin{thebibliography}{ABSR15}

\bibitem[ABSR15]{andrade2015}
J.~C. Andrade, L.~Bary-Soroker, and Z.~Rudnick.
\newblock Shifted convolution and the {T}itchmarsh divisor problem over
  {$\mathbb{F}_q[t]$}.
\newblock {\em Philos. Trans. Roy. Soc. A}, 373(2040):20140308, 18, 2015.

\bibitem[ABT93]{arratia1993}
Richard Arratia, A.~D. Barbour, and Simon Tavar\'{e}.
\newblock On random polynomials over finite fields.
\newblock {\em Math. Proc. Cambridge Philos. Soc.}, 114(2):347--368, 1993.

\bibitem[AP19]{afshar2019}
Ardavan Afshar and Sam Porritt.
\newblock The function field {S}athe-{S}elberg formula in arithmetic
  progressions and `short intervals'.
\newblock {\em Acta Arith.}, 187(2):101--124, 2019.

\bibitem[BSG18]{barysoroker2018}
Lior Bary-Soroker and Ofir Gorodetsky.
\newblock Roots of polynomials and the derangement problem.
\newblock {\em Amer. Math. Monthly}, 125(10):934--938, 2018.

\bibitem[Car82]{car1982}
Mireille Car.
\newblock Factorisation dans {$F_{q}[X]$}.
\newblock {\em C. R. Acad. Sci. Paris S\'{e}r. I Math.}, 294(4):147--150, 1982.

\bibitem[Coh70]{cohen1970}
Stephen~D. Cohen.
\newblock The distribution of polynomials over finite fields.
\newblock {\em Acta Arith.}, 17:255--271, 1970.

\bibitem[Gon42]{gontcharoff1942}
W.~Gontcharoff.
\newblock Sur la distribution des cycles dans les permutations.
\newblock {\em C. R. (Doklady) Acad. Sci. URSS (N.S.)}, 35:267--269, 1942.

\bibitem[Gor17]{gorodetsky2017}
Ofir Gorodetsky.
\newblock A polynomial analogue of {L}andau's theorem and related problems.
\newblock {\em Mathematika}, 63(2):622--665, 2017.

\bibitem[Hwa95]{hwang1995}
Hsien-Kuei Hwang.
\newblock Asymptotic expansions for the {S}tirling numbers of the first kind.
\newblock {\em J. Combin. Theory Ser. A}, 71(2):343--351, 1995.

\bibitem[Hwa98]{hwang1998}
Hsien-Kuei Hwang.
\newblock A {P}oisson {$\ast$} negative binomial convolution law for random
  polynomials over finite fields.
\newblock {\em Random Structures Algorithms}, 13(1):17--47, 1998.

\bibitem[{Lan}09]{landau1909}
E.~{Landau}.
\newblock {Handbuch der Lehre von der Verteilung der Primzahlen. Erster Band.}
\newblock {Leipzig u. Berlin: B. G. Teubner. X + 564 S. (1909).}, 1909.

\bibitem[MU05]{mitzenmacher2005}
Michael Mitzenmacher and Eli Upfal.
\newblock {\em Probability and computing}.
\newblock Cambridge University Press, Cambridge, 2005.
\newblock Randomized algorithms and probabilistic analysis.

\bibitem[MW58]{moser1958}
L.~Moser and M.~Wyman.
\newblock Asymptotic development of the {S}tirling numbers of the first kind.
\newblock {\em J. London Math. Soc.}, 33:133--146, 1958.

\bibitem[Sat53]{sathe1}
L.~G. Sathe.
\newblock On a problem of {H}ardy on the distribution of integers having a
  given number of prime factors. {I}.
\newblock {\em J. Indian Math. Soc. (N.S.)}, 17:63--82, 1953.

\bibitem[Sel54]{selberg1954}
Atle Selberg.
\newblock Note on a paper by {L}. {G}. {S}athe.
\newblock {\em J. Indian Math. Soc. (N.S.)}, 18:83--87, 1954.

\bibitem[SS03]{stein2003}
Elias~M. Stein and Rami Shakarchi.
\newblock {\em Complex analysis}, volume~2 of {\em Princeton Lectures in
  Analysis}.
\newblock Princeton University Press, Princeton, NJ, 2003.

\bibitem[Sta99]{stanley1999}
Richard~P. Stanley.
\newblock {\em Enumerative combinatorics. {V}ol. 2}, volume~62 of {\em
  Cambridge Studies in Advanced Mathematics}.
\newblock Cambridge University Press, Cambridge, 1999.
\newblock With a foreword by Gian-Carlo Rota and appendix 1 by Sergey Fomin.

\bibitem[Ten15]{tenenbaum2015}
G\'{e}rald Tenenbaum.
\newblock {\em Introduction to analytic and probabilistic number theory},
  volume 163 of {\em Graduate Studies in Mathematics}.
\newblock American Mathematical Society, Providence, RI, third edition, 2015.
\newblock Translated from the 2008 French edition by Patrick D. F. Ion.

\bibitem[Tur34]{turan1934}
Paul Tur\'{a}n.
\newblock On a {T}heorem of {H}ardy and {R}amanujan.
\newblock {\em J. London Math. Soc.}, 9(4):274--276, 1934.

\bibitem[vLW01]{vanlint2001}
J.~H. van Lint and R.~M. Wilson.
\newblock {\em A course in combinatorics}.
\newblock Cambridge University Press, Cambridge, second edition, 2001.

\bibitem[War93]{warlimont1993}
R.~Warlimont.
\newblock Arithmetical semigroups. {IV}. {S}elberg's analysis.
\newblock {\em Arch. Math. (Basel)}, 60(1):58--72, 1993.

\end{thebibliography}

\Addresses

\end{document}